\documentclass[oupthm]{CUP-JNL-BCM}%

\usepackage{graphicx}
\usepackage{multicol,multirow}
\usepackage{amsmath,amssymb,amsfonts}
\usepackage{mathrsfs}
\usepackage{rotating}
\usepackage{appendix}
\usepackage[numbers]{natbib}

\theoremstyle{oupplain}
\newtheorem{theorem}{Theorem}[section]
\newtheorem{lemma}[theorem]{Lemma}
\newtheorem{proposition}[theorem]{Proposition}

\newtheorem{definition}[theorem]{Definition}
\theoremstyle{oupremark}
\newtheorem{remark}[theorem]{Remark}
\newtheorem{example}[theorem]{Example}
\newtheorem{question}[theorem]{Question}
\theoremstyle{oupproof}
\newtheorem{proof}{Proof}

\numberwithin{equation}{section}

\articletype{RESEARCH ARTICLE}
\begin{document}

\begin{Frontmatter}

\title[Nuclear Toeplitz Operators between Fock Spaces]{Nuclear Toeplitz Operators between Fock Spaces}

\author{Tengfei Ma}
\author{Yufeng Lu}
\author{Chao Zu\thanks{Corresponding author.}}

\address{\orgname{School of Mathematics Sciences, Dalian University of Technology}, \orgaddress{\street{Dalian}, \state{Liaoning}, \postcode{116024, P. R. China}}\email{503304283@qq.com}, \email{lyfdlut@dlut.edu.cn}, \email{zuchao@dlut.edu.cn}}



\abstract{We characterize the nuclearity of Toeplitz operators $T_\mu: F_\alpha^p \to F_\alpha^q$
with Borel measure symbols for $1\leq p,q\leq \infty$. For positive measures $\mu$ and $q\leq p$, we provide necessary and sufficient conditions in terms of the Berezin transform and establish a rigidity property for nuclearity across this range. In the case $p<q$, we obtain separate necessary and sufficient conditions, indicating that the Berezin transform alone is insufficient for a complete characterization. Our results extend to Fock spaces on $\mathbb{C}^n$.}

\end{Frontmatter}
\section{Introduction}
Toeplitz operators on Fock spaces constitute a fundamental class of non-selfadjoint operators in complex analysis and operator theory and have attracted a great deal of interest. The theory of Toeplitz operators
on Fock spaces has been extensively developed in the Hilbert space setting. For the
classical Fock space $F_\alpha^2$,  Toeplitz operators with bounded or measure-valued symbols
have been studied from various perspectives, including boundedness, compactness,
and Schatten class membership, see references \cite{MR2610379,MR882716,MR1277446,MR4107813}.  In particular,
Isralowitz and Zhu \cite{MR2609242} obtained complete characterizations of boundedness, compactness, and Schatten class properties of Toeplitz operators with positive Borel measure symbols.

Beyond the Hilbert space situation, Hu and Lv \cite{MR2819157,MR3248473} systematically investigated Toeplitz operators acting between different Fock spaces $F_{\alpha}^{p}$ and $F_{\alpha}^{q}$ for $0< p,q < \infty$. They characterized bounded and compact Toeplitz operators with positive Borel measure symbols in terms of averaging functions and Berezin-type transforms. A notable feature of their results is a rigidity phenomenon specific to Fock spaces: boundedness (or compactness) from $F_{\alpha}^{p}$ to $F_{\alpha}^{q}$ for some $p \leq q$ automatically implies boundedness (or compactness) for all such $q$.  The groundbreaking work of Hu and Lv has stimulated extensive scholarly investigations into Toeplitz and Hankel operators across diverse weighted Fock spaces \cite{MR3392258,MR3157978,MR2934601}.

In contrast to the well-developed theory of bounded and compact Toeplitz
operators, considerably less is known about Toeplitz operators belonging to finer
operator ideals on Fock spaces, especially in the non-Hilbertian case. Among these
ideals, the class of nuclear operators plays a distinguished role in Banach space
operator theory.  The aim of this paper is to study nuclear Toeplitz operators acting between Fock spaces.  Our first main result is presented as follows:

\begin{theorem}\label{th1.1}
        Let $1\leq q\leq p<\infty, 1\leq s\leq\infty$, and let $\mu\geq0$ satisfy (\ref{eq2}); then the following statements are equivalent
    \begin{enumerate}
    \item[(1)] $T_\mu\in\mathcal{N}(F_\alpha^p,F_\alpha^q)$;
    \item[(2)] $T_\mu\in\mathcal{N}(f_\alpha^\infty,F_\alpha^s)$;
    \item[(3)] $\mu(\mathbb{C})<\infty$.
\end{enumerate}
Moreover,
\begin{equation}\label{eq1}
    \|T_\mu\|_{\mathcal{N}}\simeq\mu(\mathbb{C}).
\end{equation}
    \end{theorem}
Since nuclear and trace-class operators coincide on Hilbert spaces, the case $p=q=2$ in our results recovers Zhu’s characterization of trace‑class Toeplitz operators in \cite{MR2609242}. Notice that Zhu's proof relies heavily on Hilbert space structure, including
orthogonal decompositions, positivity, and the classical Schatten class theory. In contrast, when $p \neq 2$, the Fock space $F^p_\alpha$ is a Banach space lacking Hilbert geometry.  The standard techniques used in the Hilbert case do not extend directly. Thus, the characterization of nuclear Toeplitz operators between Banach Fock spaces requires a different approach based on operator ideals. We emphasize that our method provides a more constructive approach to approximating $T_\mu$ with nuclear operators compared to Zhu's proof.

Our main results provide a detailed description of nuclear Toeplitz operators in different parameter regimes. When $ 1\leq q\leq p<\infty $ and $ \mu\geq 0 $, we obtain necessary and sufficient conditions for $ T_{\mu} $ to be nuclear in terms of the integrability of the Berezin transform of $ \mu $. As a consequence, we show that nuclearity in this range exhibits a rigidity property analogous to that observed for boundedness: if $ T_{\mu} $ is nuclear from $ F_{\alpha}^{p} $ to $ F_{\alpha}^{q} $ for some $ q\leq p $, then it is nuclear for all such $ q $.

When $ p<q $, the situation becomes more subtle. In this case, we establish separate necessary and sufficient conditions for nuclearity, showing that the Berezin transform alone does not yield a complete characterization. This reflects an inherent asymmetry between the domain and target spaces in the nuclear setting and distinguishes the nuclear theory from previously studied mapping properties of Toeplitz operators on Fock spaces. We will discuss this problem in Chapter 3.

Berger and Coburn \cite{MR1277446} showed that the set of all Toeplitz operators $T_\varphi$ with compact support continuous symbol is trace-norm dense in the trace ideal of $F_\alpha^2$, and is operator-norm dense in $\mathcal{K}(F_\alpha^2)$. Our second main result generalizes the aforementioned finding :
\begin{theorem}
    \label{th1.2}
    Let $C$ denote the set of all Toeplitz operators $T_{\varphi}$ with continuous, compactly supported symbols $\varphi$ on $\mathbb{C}$, and let $1<p,q<\infty$. Then:
\begin{enumerate}
    \item [(1)] $C$ is dense in $\mathcal{N}(F^{p}_{\alpha},F^{q}_{\alpha})$ with respect to the nuclear norm.
    \item[(2)] $C$ is norm-dense in $\mathcal{K}(F^{p}_{\alpha},F^{q}_{\alpha})$.
\end{enumerate}
\end{theorem}

The proofs rely on tools from Banach space operator theory, including the duality properties of nuclear operators, approximation and Radon-Nikodym properties of Fock spaces, which we will present in Section 2.

The paper is organized as follows. In Section 2, we recall basic definitions and preliminary results concerning nuclear operators and Fock spaces. Section 3 is devoted to the characterization of nuclear Toeplitz operators between Fock spaces and contains our main results. In Section 4, we present several trace formulas and establish the density theorem.

\section{Preliminary}
\subsection{Fock Space}

Let $ \alpha>0 $ and $ p>0 $, the notation $ L^{p}_{\alpha} $ refers to the space of all Lebesgue measurable functions $ f $ on $ \mathbb{C} $ for which $ f(z)e^{-\alpha|z|^{2}/2}\in L^{p}(\mathbb{C},\mathrm{d}A) $, where $ \mathrm{d}A $ is the Lebesgue area measure on $ \mathbb{C} $. For $ f\in L^{p}_{\alpha} $, its norm is given by

$$ \|f\|_{\alpha,p}^{p}=\frac{p\alpha}{2\pi}\int_{\mathbb{C}}\left|f(z)e^{-\frac{\alpha}{2}|z|^{2}}\right|^{p}dA(z). $$
 Correspondingly, for $ \alpha>0 $ and $ p=\infty $, we denote by $ L_{\alpha}^{\infty} $ the space of measurable functions $ f $ such that

$$ \|f\|_{\infty,\alpha}=\operatorname*{ess\, sup}\left\{|f(z)|e^{-\alpha|z|^{2}/2}:z\in\mathbb{C}\right\}<\infty. $$
The Fock space $ F_{\alpha}^{p} $ is defined as the subspace of $ L_{\alpha}^{p} $ consisting of entire functions. It is straightforward to verify that $ F_{\alpha}^{p} $ becomes a Banach space when $ 1\leq p\leq\infty $. Let $ f_{\alpha}^{\infty} $ be the closed subspace of $ F_{\alpha}^{\infty} $ consisting of $f$ satisfying

$$ \lim_{|z|\to\infty}f(z)e^{-\frac{\alpha}{2}|z|^{2}}=0. $$
When $ p=2 $ , the space $ F_{\alpha}^{2} $ is a Hilbert space equipped with the inner product

$$ \langle f,g\rangle_{F_{\alpha}^{2}}=\frac{\alpha}{\pi}\int_{\mathbb{C}}f(z)\overline{g(z)}e^{-\alpha|z|^{2}}\,dA(z). $$
 As a reproducing kernel Hilbert space, its reproducing kernel is $ K_{\alpha}(z,w)=e^{\alpha z\overline{w}} $. The normalized reproducing kernel at z is denoted by $ k_{z}(w)=e^{\alpha\overline{z}w-\frac{\alpha}{2}|z|^{2}} $ , which is a unit vector in $ F_{\alpha}^{p} $ for every $ 0<p\leq\infty $ .
For $ 1\leq p<\infty $ , let $ p^{\prime} $ be the conjugate exponent satisfying $ \frac{1}{p}+\frac{1}{p^{\prime}}=1 $. We can identify the dual of $ F_{\alpha}^{p} $ with $ F_{\alpha}^{p^{\prime}} $ , and the dual of $ f_{\alpha}^{\infty} $ with $ F_{\alpha}^{1} $ , via the integral pairing
$$ \langle f,g\rangle_{F_{\alpha}^{2}}=\frac{\alpha}{\pi}\int_{\mathbb{C}}f(z)\overline{g(z)}e^{-\alpha|z|^{2}}\,dA(z), $$
 and under this identification, the norms on the dual spaces and the Fock spaces are equivalent (see \cite[Corollary 2.25, Theorem 2.26]{MR2934601}).

Given $ \alpha>0 $ , consider the Gaussian measure

$$ d\lambda_{\alpha}(z)=\frac{\alpha}{\pi}e^{-\alpha|z|^{2}}\,dA(z). $$
 Define the integral operator $ P_{\alpha} $ on $ L_{\alpha}^{p} $ by

$$ P_{\alpha}f(z)=\int_{\mathbb{C}}K_{\alpha}(z,w)f(w)\,d\lambda_{\alpha}(w),\quad f\in L_{\alpha}^{p}. $$
 It is known that $ P_{\alpha}:L_{\alpha}^{p}\to F_{\alpha}^{p} $ is bounded for all $ p $ , and it acts as the identity on $ F_{\alpha}^{p} $ , that is, $ P_{\alpha}f=f $ for $ f\in F_{\alpha}^{p} $ (see \cite[ Corollary 2.22]{MR2934601}).

For a function $ \varphi\in L^{\infty}(\mathbb{C}) $ , the associated Toeplitz operator $ T_{\varphi}:F_{\alpha}^{2}\to F_{\alpha}^{2} $ is defined by

$$ T_{\varphi}(f)=P_{\alpha}(\varphi f),\quad f\in F_{\alpha}^{2}. $$
 More generally, if $ \mu $ is a complex Borel measure on $ \mathbb{C} $ , the Toeplitz operator $ T_{\mu} $ is given by

$$ T_{\mu}(f)(z)=\frac{\alpha}{\pi}\int_{\mathbb{C}}K_{\alpha}(z,w)f(w)e^{-\alpha|w|^{2}}\,d\mu(w),\quad z\in\mathbb{C}, $$
 provided that the integrability condition
\begin{equation}
     \int_{\mathbb{C}}|K_{\alpha}(z,w)|e^{-\alpha|w|^{2}}\,d|\mu|(w)<\infty\label{eq2}
\end{equation}
 holds for every $ z\in\mathbb{C} $. Let
\begin{equation}
    \label{eq3}
    K=\operatorname{span}\{K_{z}:z\in\mathbb{C}\}.
\end{equation}
The set $K$ is dense in $ F_{\alpha}^{p} $ for every $ 1\leq p<\infty $ and also in $ f_{\alpha}^{\infty} $. Under assumption \eqref{eq2}, the operator $ T_{\mu} $ is well-defined on $K$. For a measure $ \mu $ satisfying \eqref{eq2}, its Berezin transform is defined as
\begin{equation}
    \label{eq4}
     \widetilde{\mu}(z)=\frac{\alpha}{\pi}\int_{\mathbb{C}}|k_{z}(w)|^{2}\mathrm{e}^{-\alpha|w|^{2}}\,d\mu(w)=\frac{\alpha}{\pi}\int_{\mathbb{C}}\mathrm{e}^{-\alpha|z-w|^{2}}\,d\mu(w).   
\end{equation}

\subsection{Nuclear Operator}

Let $X$ and $Y$ be complex Banach spaces. The Banach space consisting of all bounded linear operators from $X$ to $Y$ is represented by $\mathcal{L}(X,Y)$. The dual space of a Banach space $X$ is denoted by $X'$. For $x\in X$ and $x'\in X'$, we write $\langle x,x'\rangle$ for the dual pairing $x'(x)$. We say that an operator $T\in\mathcal{L}(X,Y)$ is nuclear provided that it admits a representation of the form

\begin{equation}\label{eq5}
T=\sum_{j=1}^{\infty}x_{j}'\otimes y_{j},
\end{equation}
for sequences $(x_{j}')\subset X'$ and $(y_{j})\subset Y$ satisfying $\sum_{j=1}^{\infty}\|x_{j}'\|_{X'}\|y_{j}\|_{Y}<\infty.$ For $x'\in X'$ and $y\in Y$, we denote by $x'\otimes y$  the rank-one operator from $X$ to $Y$ defined by
$$ x'\otimes y(x)=\langle x,x'\rangle y,\quad x\in X. $$
The nuclear norm of $T$ is given by

$$ \|T\|_{N}:=\inf\sum_{i=1}^{\infty}\|x_{i}^{\prime}\|_{X'}\|y_{i}\|_{Y}, $$
where the infimum in this definition is taken over all representations of the type \eqref{eq5}. We write $\mathcal{N}(X,Y)$ for the class of all nuclear operators mapping $X$ into $Y$. Equipped with the norm $\|\cdot\|_{N}$, it becomes a Banach space. In the case where $E$ and $F$ coincide, we write $\mathcal{N}(E)$ for $\mathcal{N}(E,E)$. 

A Banach space $E$ is said to have the approximation property if for every compact set $K\subset E$ and every $\varepsilon>0$, there exists a finite-rank operator $L: E\to E$ such that
$$ \|x-Lx\|<\epsilon\quad\text{for all }x\in K. $$
For a Banach space $E$ with the approximation property, the nuclear trace of an operator $T\in\mathcal{N}(E)$ can be defined by

$$ \mathrm{Tr}_{E}(T)=\sum_{j=1}^{\infty}\langle x_{j},x_{j}^\prime\rangle, $$
where $T=\sum_{j=1}^{\infty}x_{j}^{\prime}\otimes x_{j}$ is any representation as in \eqref{eq5}. It follows from \cite{MR1744872} Theorem \uppercase\expandafter{\romannumeral5}.1.2,  that the value of the trace does not depend on the particular representation
chosen. Standard references for the general theory of operator ideals include \cite{MR1209438,MR1342297,MR1744872,MR2300779}.

\begin{definition}
\label{de2.1}
Let $X$ be a Banach space, and let $(\Omega,\Sigma,\mu)$ be any finite measure space, if  for every $\mu$-continuous vector measure $G:\Sigma\to X$ of bounded variation, we can find a function $g\in L^{1}(\mu,X)$ for which

$$ G(E)=\int_{\Omega}g\,d\mu\quad\text{holds for all }E\in\Sigma, $$
we say that $X$ has the Radon-Nikodym property.

\end{definition}

\begin{proposition}\label{pro2.2}
For $1<p<\infty$, the Fock space $F_{\alpha}^{p}$ is a Banach space that enjoys both the approximation property and the Radon-Nikodym property.
\end{proposition}
\begin{proof}
    The approximation property follows from  \cite[Proposition 2.1]{MR4107813}. The Radon-Nikodym property holds because $F_\alpha^p$ is reflexive(see \cite[Corollary 3.13]{MR453964}).
\end{proof}

Let $X$ be a Banach space, if $x\in X$, define $x^{\prime\prime}: X^\prime\rightarrow\mathbb{C}$ by
\begin{equation*}
    x^{\prime\prime}(x^\prime)=x^\prime(x)
\end{equation*}
for every $x^\prime\in X^\prime$. The map $\tau: X\rightarrow X^{\prime\prime}$ defined by $\tau(x)=x^{\prime\prime}$ for all $x\in X$ is called the canonical map. Recall that $\tau$ is an isometry from $X$ onto $X^{\prime\prime}$ when $X$ is reflexive.

\begin{theorem}\label{th2.3}
 Let X and Y be Banach spaces such that Y is reflexive. The adjoint map $T\rightarrow T^{*}$ is an isometry from $\mathcal{N}(X,Y)$ onto $\mathcal{N}(Y^{\prime},X ^{\prime})$.
\end{theorem}

\begin{proof}
 Let $T=\sum_{j=1}^{\infty}x_{j}^\prime \otimes y_{j} \in \mathcal{N}(X,Y)$, where $(x_{j}^\prime)\in X^\prime, (y_j)\in Y$.  For any $y^\prime \in Y^{'}, x \in X$, we have
    \begin{align*}
        T^{*}(y^\prime)(x)&=y^\prime(Tx)\\
       & =y^\prime(\sum_{j=1}^{\infty}x_{j}^\prime(x)y_{j})\\
       &=\sum_{j=1}^{\infty}x_{j}^\prime(x)y^\prime(y_{j})   \\                             &=\sum_{j=1}^{\infty}y_{j}^{\prime\prime}(y^\prime)x_{j}^\prime(x).
    \end{align*}
     Then $T^{*}=\sum_{j=1}^{\infty}y_{j}^{\prime\prime}\otimes x_{j}^\prime$, and $\|T^{*}\|_{\mathcal{N}}\leq\sum_{j=1}^{\infty}\|y_{j}^{\prime\prime}\|_{Y^{\prime\prime}}\|x_{j}^\prime\|_{X^{\prime}}=\sum_{j=1}^{\infty}\|y_{j}\|_{Y}\|x_{j}^\prime\|_{X^{\prime}}$. Taking the infimum over all possible representations (\ref{eq5}), we obtain $\|T^{*}\|_{\mathcal{N}}\leq\|T\|_{\mathcal{N}}$. On the other hand, let $K\in\mathcal{N}(Y^{\prime},X^{\prime})$. Since $Y$ is reflexive, there exist sequences $(a_j^\prime)\in X^\prime$ and $(b_j)\in Y$ such that $K=\sum_{j=1}^{\infty}b_{j}^{\prime\prime}\otimes a_{j}^\prime$. Let $L=\sum_{j=1}^{\infty}a_{j}^\prime\otimes b_j\in\mathcal{N}(X,Y)$. It is easy to check that $L^*=K$ and $\|L\|_\mathcal{N}\leq\|K\|_\mathcal{N}$. This proves the desired result.
\end{proof}

For the dual space of compact operators and nuclear operators, we have the following characterization.

%Proposition  {\slshape(Proposition)} Proposition

\begin{proposition}\label{pro2.4}\upshape(\cite[Proposition 16.7]{MR1209438})
Let \( E \) and \( F \) be Banach spaces such that \( E^{\prime \prime} \) or \( F^{\prime} \) has the approximation property and: \( E^{\prime \prime} \) or \( F^{\prime} \) has the Radon-Nikodým property. Then
\begin{align*}
    (\mathcal{K}(E, F))^{\prime}    &= \mathcal{N}\left(E^{\prime}, F^{\prime}\right)                              && \text{isometrically}, \\
    (\mathcal{K}(E, F))^{\prime \prime} &= \left(\mathcal{N}\left(E^{\prime}, F^{\prime}\right)\right)^{\prime} 
        = \mathcal{L}\left(E^{\prime \prime}, F^{\prime \prime}\right)  && \text{isometrically}. 
\end{align*}
The duality brackets are given by
\begin{align*}
    \langle S, T\rangle_{\mathcal{K}, \mathcal{N}} &= \sum_{n=1}^{\infty} \left\langle S^{*} y_{n}^{\prime}, x_{n}^{\prime \prime} \right\rangle_{E^{\prime}, E^{\prime \prime}} \\
    \langle T, S\rangle_{\mathcal{N}, \mathcal{L}} &= \sum_{n=1}^{\infty} \left\langle y_{n}^{\prime}, S x_{n}^{\prime \prime} \right\rangle_{E^{\prime}, E^{\prime \prime}}
\end{align*}
if \( T = \sum_{n=1}^{\infty} x_{n}^{\prime \prime} \otimes y_{n}^{\prime} \in \mathcal{N}\left(E^{\prime}, F^{\prime}\right) \).
\end{proposition}

\begin{proposition}
    \label{pro2.5}
    Let $1<p,q<\infty$, then
    \begin{equation*}
        \mathcal{N}(F_{\alpha}^{p},F_{\alpha}^{q})^{\prime}=\mathcal{L}(F_{\alpha}^{q},F_{\alpha}^{p}),\quad \mathcal{K}(F_{\alpha}^{p},F_{\alpha}^{q})^{\prime}=\mathcal{N}(F_{\alpha}^{q},F_{\alpha}^{p})
    \end{equation*}
    where the duality is induced by the trace map:
    \begin{align*}
         \langle S, T\rangle&=\mathrm{Tr}(ST)=\mathrm{Tr}(TS)\\
         \langle A, B\rangle&=\mathrm{Tr}(AB)=\mathrm{Tr}(BA)
    \end{align*}
    where $S\in\mathcal{N}(F_{\alpha}^p,F_{\alpha}^{q}) $, $T\in\mathcal{L}(F_{\alpha}^{q},F_{\alpha}^{p})$, $A\in\mathcal{K}(F_{\alpha}^{p},F_{\alpha}^{q})$, $B\in\mathcal{N}(F_{\alpha}^{q},F_{\alpha}^{p})$.
\end{proposition}
\begin{proof}
   We use $E\cong F$ to denote that the two Banach spaces $E$ and $F$ are isometrically isomorphic. Then, we have
   \begin{align*}
       \mathcal{N}(F_{\alpha}^{p},F_{\alpha}^{q})^{\prime}&\cong\mathcal{N}((F_{\alpha}^{q})^\prime,(F_{\alpha}^{p})^\prime)^{\prime}\\
       &\cong\mathcal{L}((F_{\alpha}^{q})^{\prime\prime},(F_{\alpha}^{p})^{\prime\prime})\\
       &\cong\mathcal{L}(F_{\alpha}^{q},F_{\alpha}^{p}).
   \end{align*}
   The first step is obtained from Theorem \ref{th2.3}; the second step from Proposition \ref{pro2.4}; and the last step follows from reflexivity.
   
     Let $S=\sum_{j=1}^{\infty} y_{j} \otimes x_{j} \in \mathcal{N}(F_{\alpha}^{p},F_{\alpha}^{q})$ and $T\in\mathcal{L}(F_{\alpha}^{q},F_{\alpha}^{p}) $; applying Proposition \ref{pro2.4}, we obtain that the duality bracket is given by 
    \begin{align*}
        \langle S, T\rangle&=\langle S^*, \tau T\tau^{-1}\rangle_{\mathcal{N}, \mathcal{L}}\\
        &=\sum_{j=1}^{\infty}\langle y_j,\tau Tx_j\rangle\\
        &=\sum_{j=1}^{\infty}\langle Tx_j,y_j\rangle\\
        &=\mathrm{Tr}(TS)\\
        &=\mathrm{Tr}(ST).
    \end{align*}

    The other proof is similar.
\end{proof}

For convenience, in the following sections, we write $A \lesssim B$ to mean that $A \leq CB$ for some positive constant $C$, whose exact value is insignificant. The relation $A \simeq B$ signifies that $A \lesssim B \lesssim A$.

\section{Nuclear Toeplitz operator}
In this section, given a complex Borel measure $\mu$, we are going to characterize when the Toeplitz operator $T_\mu$ belongs to the nuclear class.

Let $1\leq p\leq \infty$ and $p^\prime$ be the conjugate exponent of $p$.  For $f\in F_\alpha^{p^\prime}$, we regard $f$ as an element of $(F_\alpha^p)^\prime$ under the pairing $\langle\cdot,\cdot\rangle_{F_\alpha^2}$, that is 
\begin{equation*}
    f(g)=\langle g,f\rangle_{F_\alpha^2}
\end{equation*}
for all $g\in F_\alpha^p$. Then we have $\|f\|_{(F_\alpha^p)^\prime}\simeq\|f\|_{\alpha,p^\prime}$.

\begin{definition}
    Let $T: F_\alpha^p\xrightarrow{}F_\alpha^q$ be a linear operator, where $1\leq p,q\leq\infty$, we define the Berezin transform of $T$  as
    \begin{equation*}
        \widetilde{T}(z)=\left\langle T k_{z}, k_{z}\right\rangle _{F_\alpha^2}, \quad z \in \mathbb{C}.
    \end{equation*}
\end{definition}

\begin{theorem}
    \label{th3.2}
    Let $1\leq p<\infty$, $1\leq q,s\leq\infty$. We have:
    \begin{enumerate}
        \item [(1)] If $T\in\mathcal{N}(f_\alpha^\infty,F_\alpha^s)$, then  $\widetilde{T}(z)\in L^r(dA)$ for all $1\leq r\leq \infty$.
        \item[(2)] If $T\in\mathcal{N}(F_\alpha^p,F_\alpha^q)$ and $p\leq q$, then $\widetilde{T}(z)\in L^r(dA)$ for all $(\frac{1}{p^\prime}+\frac{1}{q})^{-1}\leq r\leq\infty$.
        \item[(3)] If $T\in\mathcal{N}(F_\alpha^p,F_\alpha^q)$ and $p>q$, then $\widetilde{T}(z)\in L^r(dA)$ for all $1\leq r\leq\infty$.   
    \end{enumerate}
    Moreover,
    \begin{equation*}
        \|\widetilde{T}\|_{L^r}\lesssim\|T\|_{\mathcal{N}}.
    \end{equation*}
\end{theorem}

\begin{proof}

    To prove (1), let $T=\sum_{j=1}^{\infty} f_{j} \otimes g_{j}$, where  $\left(g_{j}\right) \subset F_\alpha^s, \left(f_{j}\right) \subset F_\alpha^1 $ be a nuclear representation. Hence
     \begin{align*}
         \widetilde{T}(z)&=\langle T k_z,k_z\rangle_{F_\alpha^2}\\
        &=\sum_{j=1}^{\infty}\langle k_z,f_j\rangle_{F_\alpha^2}\langle g_j ,k_z\rangle_{F_\alpha^2}\\  
        &=\sum_{j=1}^{\infty}\overline{f_j(z)}g_j(z)e^{-\alpha|z|^2}.
   \end{align*}
  
   Let $1\leq r\leq \infty$, we can find $m\geq1, n\geq s$ such that $\frac{1}{m}+\frac{1}{n}=\frac{1}{r}$. It follows from \cite[Theorem 2.10]{MR2934601}, we have $f_j\in F_\alpha^1\subset F_\alpha^m$, $g_j\in F_\alpha^s\subset F_\alpha^n$, and

   \begin{equation*}
       \|f_j\|_{\alpha,m}\lesssim\|f_j\|_{\alpha,1},\|g_j\|_{\alpha,n}\lesssim\|g_j\|_{\alpha,s}.
   \end{equation*}
  Then by Hölder’s inequality, we have
        \begin{align*}
            \|\widetilde{T}\|_{L^r}&\leq\sum_{j=1}^{\infty}\|{f_j(z)}e^{-\frac{\alpha}{2}|z|^2}\|_{L^m}\|g_j(z)e^{-\frac{\alpha}{2}|z|^2}\| _{L^n}\\
             &\lesssim \sum_{j=1}^{\infty}\|f_j\|_{\alpha,m}\|g_j\| _{\alpha,n}\\
             &\lesssim \sum_{j=1}^{\infty}\|f_j\|_{\alpha,1}\|g_j\| _{\alpha,s}\\
             &\lesssim \sum_{j=1}^{\infty}\|{f_j}\|_{(f_\alpha^\infty)^\prime}\|g_j\| _{\alpha,s}.
    \end{align*}
Now, taking the inﬁmum over all possible representations (\ref{eq5}) yields
    \begin{equation*}
        \|\widetilde{T}\|_{L^r}\lesssim\|T\|_{\mathcal{N}}<\infty.
    \end{equation*}

  The proof of (2) and (3) is similar.      
\end{proof}

\begin{lemma}
    \label{le3.3}
    Let $1\leq p\leq\infty$, $z\rightarrow k_z$ is norm-continuous on $F_\alpha^p$.
\end{lemma}
\begin{proof}
    Fix $z_0\in \mathbb{C}$, suppose $|z-z_0|<1$ and $1\leq p<\infty$. Then
    \begin{equation*}
        \|k_z-k_{z_{0}}\|_{\alpha,p}^p=\frac{\alpha p}{2\pi}\int_\mathbb{C}|(k_z(w)-k_{z_0}(w))e^{-\frac{\alpha}{2}|w|^2}|^pdA(w).
    \end{equation*}
    Since $|z-z_0|<1$, there exists a constant $C>0$ such that $|(k_z(w)-k_{z_0}(w))e^{-\frac{\alpha}{2}|w|^2}|\leq Ce^{-\frac{\alpha}{4}|w|^2}$.
    According to the dominated convergence theorem, the conclusion holds for $1\leq p < \infty $. When $p=\infty$, the conclusion follows from the fact $\|f\|_{\alpha,\infty}\leq\|f\|_{\alpha,p}$ for any $1\leq p < \infty $.
\end{proof}

For a fixed positive number $r$, let 
$$ r\mathbb{Z}^{2}=\{nr+imr:n,m\in \mathbb{Z}\}. $$
Then, we define
$$ S_{r}=\{z=x+iy:-r/2\leq x<r/2,-r/2\leq y<r/2\}. $$
Clearly, we have the disjoint decomposition for $\mathbb{C}$:
$$ \mathbb{C}=\bigcup\limits_{w\in r\mathbb{Z}^{2}}(S_{r}+w), $$

Finally, let $\{a_{n}\}$ denote any fixed enumeration of the square lattice $r\mathbb{Z}^{2}$ as a sequence.

\begin{theorem}
    \label{th3.4}
   Let \(\mu\) be a complex Borel measure on \(\mathbb{C}\) with \(|\mu|(\mathbb{C})<\infty\). Then \(T_\mu\) belongs to both \(\mathcal{N}(f_\alpha^\infty, F_\alpha^s)\) and \(\mathcal{N}(F_\alpha^p, F_\alpha^q)\) for \(1\leq p<\infty\) and \(1\leq q,s\leq\infty\). Moreover,
\[
\|T_\mu\|_{\mathcal{N}} \leq \frac{\alpha}{\pi} \, |\mu|(\mathbb{C}).
\]
\end{theorem}
    \begin{proof}
       We first proof $T_\mu\in\mathcal{N}(f_\alpha^\infty, F_\alpha^1)$. Let $z\in \mathbb{C}$, we consider the operator $k_z\otimes k_z\in\mathcal{N}(f_\alpha^\infty,F_\alpha^1) $, i.e.
        \begin{equation*}
            (k_z\otimes k_z)f=\langle f,k_z\rangle _{F_\alpha^2}k_z.
        \end{equation*}
        It is obvious $\|k_z\otimes k_z\|_\mathcal{N}\leq1$ for all $z\in \mathbb{C}$.
For every $r>0$, note that 
$$\sum_{k=1}^\infty |\mu(S_r+a_k)|\leq\sum_{k=1}^\infty|\mu|(S_r+a_k)=|\mu|(\mathbb{C})<\infty,$$ 
thus we can define
    \begin{equation*}
         T_r=\sum_{k=1}^\infty \mu(S_r+a_k)k_{a_k}\otimes k_{a_k} \in \mathcal{N}{(f_\alpha^\infty,F_\alpha^1)}.
    \end{equation*}
    
    Let $r_n=\frac{1}{2^n}$, fix any $\epsilon>0$, since $|\mu|(\mathbb{C})<\infty$,  there exists $R>0$ such that 
    \begin{equation*}
        |\mu|(\mathbb{C}\setminus \overline{ B(0,R)}<\epsilon.
    \end{equation*}
    According to Lemma \ref{le3.3}, there exists a $\delta>0$ such that 
    \begin{equation*}  
          \|k_z-k_w\|_{\alpha,1}<\epsilon
    \end{equation*} 
   whenever $|z-w|<\delta $, and  $z,w\in \overline{B(0,R+1)}$. 
    
    Since$ \{r_n\}$ is a monotonically decreasing sequence tending to zero, there exists an integer $N_0$ such that for all $ N > N_0 $, the set ${\{(S_{r_N}+a_k^{(N)})}\}$ satisfies the following property: if $S_{r_N}+a_k^{(N)} \cap\overline{ B(0, R)} \neq\emptyset$, then $\overline{(S_{r_N}+a_k^{(N)})}$ is contained in $\overline{B(0, R+1)}$. 
    
    On the one hand, there exists an integer $N_1$ such that $r_N < \sqrt{2} \delta$ for all $N > N_1$. Fix an $N > \max(N_0, N_1)$. For any $N' > N$, recall that in a Banach space, any absolutely summable sequence is unconditionally summable \cite[Proposition 1.1]{MR1342297}. We have,
   \begin{align*}
       &T_{r_N}-T_{r_{N^\prime}}=\sum_{k=1}^\infty \mu(S_{r_N}+a_k^{(N)})k_{a_k^{(N)}}\otimes k_{a_k^{(N)}}-\sum_{j=1}^\infty \mu(S_{r_{N^{\prime}}}+a_j^{(N^\prime)})k_{a_j^{(N^\prime)}}\otimes k_{a_j^{(N^\prime)}}\\
       &=\sum_{k=1}^\infty\sum_{j=1}^\infty \mu[(S_{r_N}+a_k^{(N)})\cap (S_{r_{N^{\prime}}}+a_j^{(N^\prime)})](k_{a_k^{(N)}}\otimes k_{a_k^{(N)}}-k_{a_j^{(N^\prime)}}\otimes k_{a_j^{(N^\prime)}})\\
       &=I_1+I_2
   \end{align*}
   where
   \begin{align*}
      & I_1=\sum_{k=1}^\infty\sum_{j=1}^\infty \mu[(S_{r_N}+a_k^{(N)})\cap (S_{r_{N^{\prime}}}+a_j^{(N^\prime)})](k_{a_k^{(N)}}-k_{a_j^{(N^\prime)}})\otimes k_{a_j^{(N^\prime)}}\\
     & I_2=\sum_{k=1}^\infty\sum_{j=1}^\infty \mu[(S_{r_N}+a_k^{(N)})\cap (S_{r_{N^{\prime}}}+a_j^{(N^\prime)})]k_{a_k^{(N)}}\otimes(k_{a_k^{(N)}}-k_{a_j^{(N^\prime)}}).\\
   \end{align*}
   
   Let 
   $$S_1=\{k\in \mathbb{N},(S_{r_N}+a_k^{(N)}) \cap \overline{B(0,R)}\neq\emptyset\},$$
   $$S_2=\{k\in \mathbb{N},(S_{r_N}+a_k^{(N)}) \cap \overline{B(0,R)}=\emptyset\}.$$
   It is easy to see that if $(S_{r_N}+a_k^{(N)})\cap (S_{r_{N^{\prime}}}+a_j^{(N^\prime)})\neq\emptyset,$ then we have $a_k^{(N)},$ $a_j^{(N^\prime)}\in \overline{(S_{r_N}+a_k^{(N)})}$, and $ |a_k^{(N)}-a_j^{(N^\prime)}|<\delta$. Thus
   \begin{align*}
     & \| \sum_{k\in S_1}\sum_{j=1}^\infty \mu[(S_{r_N}+a_k^{(N)})\cap (S_{r_{N^{\prime}}}+a_j^{(N^\prime)})]k_{a_k^{(N)}}\otimes(k_{a_k^{(N)}}-k_{a_j^{(N^\prime)}})\|_\mathcal{N}\\
     &\leq \sum_{k\in S_1}\sum_{j=1}^\infty |\mu|[(S_{r_N}+a_k^{(N)})\cap (S_{r_{N^{\prime}}}+a_j^{(N^\prime)})]\|(k_{a_k^{(N)}}-k_{a_j^{(N^\prime)}})\|_{\alpha,1}\| k_{a_j^{(N^\prime)}}\|_{(f_\alpha^\infty)^\prime}\\
     &\leq |\mu|(\mathbb{C})\epsilon.
   \end{align*}
   On the other hand,
   \begin{align*}
        & \| \sum_{k\in S_2}\sum_{j=1}^\infty \mu[(S_{r_N}+a_k^{(N)})\cap (S_{r_{N^{\prime}}}+a_j^{(N^\prime)})]k_{a_k^{(N)}}\otimes(k_{a_k^{(N)}}-k_{a_j^{(N^\prime)}})\|_\mathcal{N}\\
        &\leq2|\mu|(\mathbb{C}\setminus \overline{ B(0,R)})\\
        &\leq2\epsilon.
   \end{align*}
   Therefore, we have $\|I_2\|_{\mathcal{N}}\leq(2+|\mu|{(\mathbb{C}}))\epsilon$. Similarly, we have
   $\|I_1\|_{\mathcal{N}}\leq C\epsilon$.
   
   Then for any $N^\prime>N$, we have 
   \begin{equation*}
       \|T_{r_N}-T_{r_N^\prime}\|_\mathcal{N}\leq\|I_1\|_{\mathcal{N}}+\|I_2\|_{\mathcal{N}}\leq C\epsilon.
   \end{equation*}
   
    Thus, $T_{r_n}$ is a Cauchy sequence in $\mathcal{N}({f_\alpha^\infty,F_\alpha^1})$. Then there is a $T\in\mathcal{N}(f_\alpha^\infty, F_\alpha^1)$,  such that $T_{r_n}\rightarrow T$ with $\|\cdot\|_\mathcal{N}$. Therefore, $T_{r_n}\rightarrow T$ with $\|\cdot\|_{op}$ (operator norm), since $\|A\|_{op}\leq\|A\|_\mathcal{N}$ for all $A\in\mathcal{N}( f_\alpha^\infty,F_\alpha^1 )$. Then it implies  that $\widetilde{T_{r_n}}\rightarrow{\widetilde{T}}$ pointwise.

   Recall that $\widetilde{\mu}(z)=\frac{\alpha}{\pi}\int_\mathbb{C}e^{-\alpha|z-w|^2} d\mu(w)$, we have
    \begin{align*}
        & | \widetilde{T_r}(z)-\frac{\pi}{\alpha}\widetilde{\mu}(z)|=|\sum_{k=1}^\infty \mu(S_r+a_k)e^{-\alpha|z-a_k|^2}-\int_\mathbb{C}e^{-\alpha|z-w|^2} d\mu(w)|\\
     &=|\sum_{k=1}^\infty\int_{S_r+a_k}e^{-\alpha|z-w|^2}-e^{-\alpha|z-a_k|^2}d\mu(w)|\\
     &\leq\sum_{k=1}^\infty\int_{S_r+a_k}|e^{-\alpha|z-w|^2}-e^{-\alpha|z-a_k|^2}|d|\mu|(w).
    \end{align*}

   Since $e^{-\alpha|z|^2}$ is uniformly continuous on $\mathbb{C}$, for any $\epsilon>0$, there exists a $r_0>0$, if $r<r_0$, we have 
   \begin{equation*}
       |e^{-\alpha|z-w|^2}-e^{-\alpha|z-a_k|^2}|<\frac{\epsilon}{|\mu|(\mathbb{C})}
   \end{equation*}
    for every $w\in S_r+a_k$. Hence, if $r<r_0$, we have
   \begin{align*}
       | \widetilde{T_r}(z)-\frac{\pi}{\alpha}\widetilde{\mu}(z)|<\epsilon.
   \end{align*}

    This implies  $\lim\limits_{r\rightarrow0}\widetilde{T_r}(z)=\frac{\pi}{\alpha}\widetilde{\mu}(z)$, hence $\widetilde{T}(z)=\frac{\pi}{\alpha}\widetilde{\mu}(z)$.

    Set $F(z,w)=T_\mu K_z(w)$ and $G(z,w)=TK_z(w)$. These functions are analytic with respect to $w$, and are anti-holomorphic with respect to $z$. Furthermore, we have
$$
F(z,z)=\frac{\alpha}{\pi}G(z,z)=e^{\alpha|z|^2}\widetilde{\mu}(z).
$$
A standard polarization theorem in several complex variables (see \cite{MR1846625}) yields $F(z,w)=\frac{\alpha}{\pi}G(z,w)$ for all $z,w\in\mathbb{C}$. Then we have $T_\mu K_z=\frac{\alpha}{\pi}TK_z$ for all $z\in\mathbb{C}$. By the density of the finite linear span of kernel functions in $f_\alpha^{\infty}$, we have $T_\mu=\frac{\alpha}{\pi}T\in\mathcal{N}(f_\alpha^{\infty},F_\alpha^{1})$.

    The above proof shows that $\frac{\alpha}{\pi}T_{r_n}\stackrel{\|\cdot\|_{\mathcal{N}}}{\longrightarrow}T_\mu$. Then we have
    \begin{equation*}
        \|T_\mu\|_{\mathcal{N}}=\frac{\alpha}{\pi}\lim\limits_{n\rightarrow\infty}\|T_{r_n}\|_{\mathcal{N}}\leq\frac{\alpha}{\pi} |\mu|(\mathbb{C}).
    \end{equation*}

       Recall that the normalized reproducing kernel \(k_z\) is a unit vector in \(F_\alpha^p\) for all \(1\leq p\leq\infty\), hence the operator \(T_r\) defined above belongs to both \(\mathcal{N}(f_\alpha^\infty, F_\alpha^s)\) and \(\mathcal{N}(F_\alpha^p, F_\alpha^q)\) for the indices under consideration. Repeating this argument completes the proof for all cases.
    \end{proof}

From the proof of Theorem \ref{th3.4}, let
\[
S = \{T_\mu : |\mu|(\mathbb{C}) < \infty\},
\]
viewed as a subset of \(\mathcal{N}(F_{\alpha}^p, F_{\alpha}^q)\). Then we have the inclusion
\[
S \subset \overline{\operatorname{span}}^{\|\cdot\|_{\mathcal{N}}} \{k_z \otimes k_z : z \in \mathbb{C}\}.
\]

  In 2020, Fulsche proved that $T_{f}\in\mathcal{N}(F_{\alpha}^{p})$ if $f\in L^{1}(\mathbb{C})$ (see \cite[lemma 2.11]{MR4107813}, thus the above Theorem can be viewed as a generalization of this result. Moreover, if $\mu$ is a positive measure, Theorem \ref{th1.1} shows that the condition $\mu(\mathbb{C})<\infty$ is also sufficient for $T_{\mu}\in\mathcal{N}(F_{\alpha}^{p},F_{\alpha}^{q})$ when $q\leq p$. Next, we give the proof of Theorem \ref{th1.1}:
  
    \begin{proof}[Proof of Theorem \ref{th1.1}]
       A direct calculation shows that $\widetilde{T_\mu}(z)=\widetilde{\mu}(z)$. Since $\mu\geq0$, we can use Fubini’s Theorem to obtain
       \begin{align*}
           \|\widetilde{\mu}\|_{L^1}&=\frac{\alpha}{\pi}\int_{\mathbb{C}}\int_{\mathbb{C}}e^{-\alpha|z-w|^2}d\mu(w)dA(z)\\
           &=\frac{\alpha}{\pi}\int_{\mathbb{C}}\int_{\mathbb{C}}e^{-\alpha|z-w|^2}dA(z)d\mu(w)\\
           &=\mu(\mathbb{C}).
       \end{align*}
       $ (1) \Rightarrow (3) $:  If $1\leq q\leq p<\infty$,  by Theorem \ref{th3.2} (3), we have $\widetilde{T_\mu}(z)=\widetilde{\mu}(z)\in L^1(\mathbb{C},dA)$, which implies $\mu(\mathbb{C})<\infty$. $(3) \Rightarrow (1)$ follows from Theorem \ref{th3.4}. The proof of $(2)\Leftrightarrow(3)$ is similar.

       Estimate (\ref{eq1}) follows from the equation $\|\widetilde{\mu}\|_{L^1}=\mu(\mathbb{C})$, Theorem \ref{th3.2} and Theorem \ref{th3.4}.
    \end{proof}

\begin{remark}
    Suppose $1\leq p<q,$ and $\mu\geq0$. If $\mu(\mathbb{C})<\infty$, by Theorem \ref{th3.4}, we have $T_\mu\in\mathcal{N}(F_\alpha^p,F_\alpha^q)$. However, we cannot use Theorem \ref{th3.2} to prove the equivalence between $T_\mu\in\mathcal{N}(F_\alpha^p,F_\alpha^q)$ and $\mu(\mathbb{C})<\infty$,  since in this case $(\frac{1}{p^\prime}+\frac{1}{q})^{-1}>1$. In fact, we will show that there exists $T\in\mathcal{N}(F_\alpha^p,F_\alpha^q)$ such that $\widetilde{T}(z)\notin L^1$.
\end{remark}

\begin{theorem}
    \label{th3.6}
    Let $1\leq p<\infty$, and let $g$ be an entire function. Then 
    $$\int_{\mathbb{C}}|f(z)\overline{g(z)}|e^{-\alpha|z|^2}dA(z)<\infty$$
     for all $f\in F_\alpha^p$ if and only if $g\in F_\alpha^{p^\prime}$.
\end{theorem}
\begin{proof}
    $\Rightarrow:$ For any $n\in \mathbb{N}^{+}$, define the truncated function
    \begin{equation*}
        g_n(z)=\chi_{\{|g(z)|\leq n\}}\cdot g(z),
    \end{equation*}
    so that
    \begin{equation*}
        \lim_{n\rightarrow\infty}g_n(z)=g(z)
    \end{equation*}
    for all $z\in\mathbb{C}$. For any measurable function $h$, define $\varphi_{h}: F_\alpha^p\rightarrow\mathbb{C}$ by
    \begin{equation*}
        \varphi_{h}(f)=\int_\mathbb{C}f(z)\overline{h(z)}e^{-\alpha|z|^2}dA(z),\quad  f\in F_\alpha^p.
    \end{equation*}
Then we have $\varphi_{g_n}\in(F_\alpha^p)^\prime$.  Fix $f\in F_\alpha^p$, since $|g_n(z)|\leq|g(z)|$, we can use the Dominated Convergence Theorem to obtain
\begin{equation*}
    \lim_{n\rightarrow\infty}\varphi_{g_n}(f)=\varphi_g(f).
\end{equation*}
Hence $\sup_n|\varphi_{g_n}(f)|<\infty$. It follows from the Principle of Uniform Boundedness that we have
\begin{equation*}
\sup_n\|\varphi_{g_n}\|=C<\infty.
\end{equation*}
Then 
\begin{equation*}
    |\varphi_g(f)|=\lim\limits_{n\rightarrow\infty}|\varphi_{g_n}(f)|\leq C\|f\|
\end{equation*}
for all $f\in F_\alpha^p$. It implies that $\varphi_g\in (F_\alpha^p)^\prime$. Since $(F_\alpha^p)^\prime$ can be identified with $F_\alpha^{p^\prime}$ under the integral pairing $\langle\cdot,\cdot\rangle_{F_\alpha^2}$, we can find $g^\prime\in F_\alpha^{p^\prime}$ such that
\begin{equation*}
    \varphi_{g-g^\prime}=0.
\end{equation*}
Suppose $(g-g^\prime)(z)=\sum\limits_{k=0}^\infty a_kz^k$. For any $n\in\mathbb{N}$ , we have 
\begin{equation*}
    \int_{\mathbb{C}}|z^n||\overline{(g-g^\prime)(z)}|e^{-\alpha|z|^2}dA(z)<\infty.
\end{equation*}
Then we can use the Dominated Convergence Theorem to obtain
\begin{align*}
    \varphi_{g-g^\prime}(z^n)&= \int_{\mathbb{C}}z^n\overline{(g-g^\prime)(z)}e^{-\alpha|z|^2}dA(z)\\
    &=lim_{R\rightarrow{\infty}}\int_{|z|<R}z^n\overline{(g-g^\prime)(z)}e^{-\alpha|z|^2}dA(z)\\
          &=lim_{R\rightarrow{\infty}} \overline{a_n}\int_{|z|<R}|z|^{2n}e^{-\alpha|z|^2}dA(z).\\
          &=\overline{a_n}\int_\mathbb{C}|z|^{2n}e^{-\alpha|z|^2}dA(z)\\
          &=0.
\end{align*}
    It implies that $a_n=0$ for all $n\in\mathbb{N}$. Then we have $g=g^\prime\in F_\alpha^{p^\prime}$.

    The other direction follows from Hölder’s inequality.
   \end{proof}
\begin{example}
    It is well-known that for $1\leq p<q\leq\infty$, we have the (proper) inclusion $F_\alpha^p\subset F_\alpha^q$, and the inclusion is proper. Theorem \ref{th3.6} tells us: We can find $f\in F_\alpha^{p^\prime}$ and $g\in F_\alpha^q$ such that
    \begin{equation*}
        \int_{\mathbb{C}}|f(z)\overline{g(z)}|e^{-\alpha|z|^2}dA(z)=\infty.
    \end{equation*}
In other words, if we define $T=f\otimes g\in \mathcal{N}(F_\alpha^p,F_\alpha^q)$, then we have $\|\widetilde{T}\|_{L^1}=\infty.$  Subsequently, we give an explicit construction of $f$ and $g$.

    If $1 < p < q < \infty$, we have $1 < q' < p' < \infty$. Let $a = \frac{1}{q} - \frac{1}{p} = \frac{1}{p'} - \frac{1}{q'}$ and define $\epsilon_{n} = n^{\frac{a}{2}}$. A sequence of indices $n_{1} < n_{2} < \cdots$ can be chosen so that $3^{-k} \leq \epsilon_{n_{k}} \leq 2^{-k}$ for all $k \in \mathbb{N}$. Fix a constant $b$ with $\sqrt{3} < b < 2$. We consider the function defined by $f(z) = \sum a_{n}z^{n}$, where its coefficients are given by
   \begin{equation*}
a_n = \begin{cases} 
b^k\sqrt{\frac{\alpha^{n_{k}}}{(n_k)!}} \, n_k^{\frac{1}{4} - \frac{1}{2p^\prime}} \, \epsilon_{n_k}, & \text{if } n = n_k \text{ for some } k, \\ 
0, & \text{otherwise}.
\end{cases}
\end{equation*}

The membership of $f$ in $F_{\alpha}^{p'}$ is verified by computing two sums. For $1 < p' \leq 2$,
\begin{equation*}
\sum_{n=0}^{\infty} |a_{n}|^{p'} \left( \frac{n!}{\alpha^{n}} \right)^{\frac{p'}{2}} n^{-\frac{p'}{4} + \frac{1}{2}} = \sum_{k=1}^{\infty} b^{p'k} \epsilon_{n_{k}}^{p'} \leq \sum_{k=1}^{\infty} \left( \frac{b^{p'}}{2^{p'}} \right)^{k} < \infty.
\end{equation*}
If $2 \leqslant p' < \infty$, then
\begin{equation*}
\sum_{n=1}^{\infty} |a_{n}|^{p} \left( \frac{n!}{\alpha^{n}} \right)^{\frac{p}{2}} n^{\frac{p}{4} - \frac{1}{2}} = \sum_{k=1}^{\infty} b^{pk} n_{k}^{\frac{p}{4} - \frac{1}{2} + \frac{p}{4} - \frac{p}{2p'}} \epsilon_{n_{k}}^{p} \leq \sum_{k=1}^{\infty} \left( \frac{b^{p}}{2^{p}} \right)^{k} < \infty.
\end{equation*}
The convergence of these series, together with Theorems 1 and 4 in \cite{MR2215157}, implies that $ f \in F_{\alpha}^{p'} $. Since $a_{n}\neq o\left(\sqrt{\frac{\alpha^{n}}{n!}}n^{\frac{1}{4}-\frac{1}{2q^{\prime}}}\right)\quad(n\rightarrow\infty), $
it follows from Corollary 5 of \cite{MR2215157} that $ f\in F_{\alpha}^{p^{\prime}}\setminus F_{\alpha}^{q^{\prime}} $.

A function $g$ is constructed analogously. Let
\begin{equation*}
b_{m}=\begin{cases}
b^{k}\sqrt{\dfrac{\alpha^{n_{k}}}{(n_{k})!}}\,n_{k}^{\frac{1}{4}-\frac{1}{2q}}\,\epsilon_{n_{k}},&\text{if }m=n_{k},\\
0,&\text{otherwise,}
\end{cases}
\end{equation*}
and define $ g(z)=\sum b_{m}z^{m}. $ It then follows that $ g\in F_{\alpha}^{q}\setminus F_{\alpha}^{p}. $

If $ \int_{\mathbb{C}}|f(z)g(z)|e^{-\alpha|z|^{2}}\mathrm{d}A(z)<\infty, $ then by the Dominated Convergence Theorem, we obtain
\begin{equation*}
\begin{aligned}
\int_{\mathbb{C}}f(z)\overline{g(z)}e^{-\alpha|z|^{2}}\mathrm{d}A(z)&=\lim_{R\rightarrow\infty}\int_{|z|<R}f(z)\overline{g(z)}e^{-\alpha|z|^{2}}\mathrm{d}A(z)\\
&=\lim_{R\rightarrow\infty}\sum_{n=0}^{\infty}\sum_{m=0}^{\infty}a_{n}\overline{b_{m}}\int_{|z|<R}z^{n}\overline{z}^{m}e^{-\alpha|z|^{2}}\mathrm{d}A(z)\\
&=\lim_{R\rightarrow\infty}\sum_{n=0}^{\infty}a_{n}\overline{b_{n}}\int_{|z|<R}|z|^{2n}e^{-\alpha|z|^{2}}\mathrm{d}A(z).
\end{aligned}
\end{equation*}
From the construction above, we have $ a_{n}\overline{b_{n}}\geq 0 $ and $ \int_{|z|<R}|z|^{2n}e^{-\alpha|z|^{2}}\mathrm{d}A(z) $ is monotonically increasing in $R$. The Monotone Convergence Theorem yields:
    \begin{align*}
\int_{\mathbb{C}}f(z)\overline{g(z)}e^{\alpha|z|^2}dA(z)&=lim_{R\rightarrow{\infty}}\sum_{n=0}^\infty a_n\overline{b_n}\int_{|z|<R}|z|^{2n}e^{-\alpha|z|^2}dA(z)\\
        &=\sum_{n=0}^\infty a_n\overline{b_n}\int_{\mathbb{C}}|z|^{2n}e^{-\alpha|z|^2}dA(z)\\
        &=\frac{\pi}{\alpha}\sum_{k=1}^\infty\gamma^{2k}\delta_{n_k}\\
        &\geq\frac{\pi}{\alpha}\sum_{k=1}^\infty(\frac{\gamma^2}{3})^k\\
        &=\infty.
    \end{align*}
But $\int_{\mathbb{C}}f(z)\overline{g(z)}\mathrm{e}^{-\alpha|z|^2}dA(z)\leq\int_{\mathbb{C}}|f(z){g(z)}|\mathrm{e}^{-\alpha|z|^2}dA(z)$. We obtain a contradiction.
\end{example}

Finally, we conclude this section with an open problem.
\begin{question}
For $1\leq p<q\leq\infty$ and $\mu\geq0$, what is the necessary and sufficient condition for $T_\mu\in\mathcal{N}(F_\alpha^p,F_\alpha^q)$?
\end{question}

\section{Trace formula and density theorem }
In this section, for an operator $T\in\mathcal{N}(F_\alpha^p)$, $ (1<p<\infty)$, we will provide some formulas for calculating its trace. Finally, we will establish a density theorem.

\begin{lemma}
\label{le4.1}
    Let $1<p<\infty$ and $T\in \mathcal{N}(F^p_\alpha)$, then we have:
    \begin{equation*}
        \mathrm{Tr}(T)=\frac{\alpha}{\pi}\int_\mathbb{C}\widetilde{T}(z)dA(z).
        \end{equation*}
        \end{lemma}
        \begin{proof}
            Let $T=\sum_{j=1}^{\infty} f_{j} \otimes g_{j}$ be a nuclear representation of $T$, where  $\left(g_{j}\right) \subset F_\alpha^p,\left(f_{j}\right) \subset F_\alpha^{p^\prime }$. Then we have $\mathrm{Tr}(T)=\sum_{j=1}^{\infty}\langle g_j,f_j\rangle=\sum_{j=1}^{\infty}\langle g_j,f_j\rangle_{F_\alpha^2}$. And the trace is independent of the choice of representation. Then we have:
            \begin{align*}
                \frac{\alpha}{\pi}\int_\mathbb{C}\widetilde{T}(z)dA(z)&=\frac{\alpha}{\pi}\int_\mathbb{C}\langle Tk_z,k_z\rangle_{F_\alpha^2}dA(z)\\
                &=\frac{\alpha}{\pi}\int_\mathbb{C}\sum_{j=1}^{\infty}\langle K_z,f_j\rangle_{F_\alpha^2}\langle g_j,K_z\rangle_{F_\alpha^2}e^{-\alpha|z|^2}dA(z)\\
                &=\sum_{j=1}^{\infty}\int_\mathbb{C}\overline{f_j(z)}g_j(z)d\lambda_\alpha(z)\\
                &=\mathrm{Tr}(T).
            \end{align*}.       
        \end{proof}

Notice that the case  $p=q=2$ of the following theorem appears in \cite[Theorem 6]{MR1277446}, but we adopt a different approach here.

     \begin{theorem}
    \label{th4.2}
    Suppose $\varphi$ is Lebesgue measurable on $\mathbb{C}$ and $S\in\mathcal{L}(F_\alpha^q,F_\alpha^p)$, where $1<p,q<\infty$. If
    \begin{enumerate}
        \item [(1)] $\varphi K_z \in L_\alpha^{p^{\prime}}$ for every $z\in \mathbb{C}$,
        \item[(2)] $T_\varphi\in\mathcal{L}(F_\alpha^p,F_\alpha^q)$,
        \item[(3)]  $T_\varphi S \in \mathcal{N}(F_\alpha^q)$,
        \item[(4)] $\int_\mathbb{C}\int_\mathbb{C}|\varphi(w)||K(z,w)||SK_z(w)|d\lambda_\alpha(z)d\lambda_\alpha(w)<\infty$.
    \end{enumerate}
  Then we have
    \begin{equation*}
        \mathrm{Tr}(T_\varphi S)=\frac{\alpha}{\pi}\int_\mathbb{C}\varphi(z)\widetilde{S}(z)dA(z).
    \end{equation*}
\end{theorem}
\begin{proof}
   Let $f\in F_\alpha^p $, we can find a sequence $f_n\in K$ such that $f_n$ converges to $f$ with $\|\|_{\alpha,p}$. The ser $K$ is defined as (\ref{eq3}). By assumption (2), we have $T_\varphi f_n \stackrel{\|\cdot\|_{\alpha,q}}{\longrightarrow}T_\varphi f$, hence $T_\varphi f_n(z)\rightarrow T_\varphi f(z)$ for every $z\in \mathbb{C}$.

Since $f_n\in K$, we have
\begin{equation*}
    T_\varphi f_n(z)=\int_\mathbb{C}f_n(w)\varphi(w)K_\alpha(z,w)d\lambda_\alpha(w)
\end{equation*}
   for all $n\in \mathbb{N}$ and $z\in\mathbb{C}$.
   Then fix $z\in\mathbb{C}$, by assumption (1), we have
   \begin{align*}
       &|\int_\mathbb{C}(f_n(w)-f(w))\varphi(w)K(z,w)d\lambda_\alpha(w)|
       \\
       &\leq \frac{\alpha}{\pi}\int_\mathbb{C}|(f_n(w)-f(w))e^{-\frac{\alpha}{2}|z|^2}||\varphi(w)K(z,w)e^{-\frac{\alpha}{2}|z|^2}|dA(w)\\
       &\leq C\|f_n-f\|_{\alpha,p}\|\varphi K_z\|_{\alpha,p^\prime}\rightarrow0(n\rightarrow\infty).
   \end{align*}
   
   Hence
    \begin{equation*}
        T_\varphi f(z)=\int_\mathbb{C}f(w)\varphi(w)K(z,w)d\lambda_\alpha(w)
    \end{equation*}
    for every $f\in F_\alpha^p$.
    Then we can use Fubini's Theorem to obtain 
    \begin{equation}\label{eq6}
        \begin{aligned}
        \mathrm{Tr}(T_\varphi S)&=\frac{\alpha}{\pi}\int_\mathbb{C}\widetilde{T_\varphi S}(z)dA(z)\\
        &=\int_\mathbb{C}\langle T_\varphi SK_z,K_z\rangle_{F_\alpha^2}d\lambda_\alpha(z)\\
        &=\int_\mathbb{C}T_\varphi SK_z(z) d\lambda_{\alpha}(z)\\
        &=\int_\mathbb{C}d\lambda_{\alpha}(z)\int_\mathbb{C}\varphi(u)K_u(z)SK_z(u)d\lambda_\alpha(u)\\
        &=\int_\mathbb{C}\varphi(u)d\lambda_\alpha(u)\int_\mathbb{C}K_u(z)SK_z(u)d\lambda_{\alpha}(z).
    \end{aligned}
    \end{equation}
      
      Let $S^\prime\in\mathcal{L}(F_\alpha^{p^\prime},F_\alpha^{q^\prime})$ be the adjoint operator of $S$ under the pairing $\langle\cdot,\cdot\rangle_{F_\alpha^2}$. Then we have
      \begin{equation*}
          \langle Sf,g\rangle_{F_\alpha^2}=\langle f,S^\prime g\rangle_{F_\alpha^2}
      \end{equation*}
      for all $f\in F_\alpha^q, g\in F_\alpha^{p^\prime}$ . Hence:
        \begin{align*}
            \int_\mathbb{C}K_u(z)SK_z(u)d\lambda_{\alpha}(z)&=\int_\mathbb{C}K_u(z)\langle SK_z,K_u\rangle_{F_\alpha^2}d\lambda_\alpha(z)\\
            &=\int_\mathbb{C}K_u(z)\langle K_z,S^\prime K_u\rangle_{F_\alpha^2}d\lambda_\alpha(z)\\
            &=\overline{\int_\mathbb{C}K_z(u)S^\prime K_u(z)d\lambda_{\alpha}(z)}\\
            &=\overline{S^\prime K_u(u)}\\
            &=e^{\alpha|u|^2}\widetilde{S}(u).
        \end{align*}
        It follows from (\ref{eq6}), we have:
        \begin{equation*}
             \mathrm{Tr}_{F_\alpha^q}(T_\varphi S)=\frac{\alpha}{\pi}\int_\mathbb{C}\varphi(u)\widetilde{S}(u)dA(u).
        \end{equation*}
        
\end{proof}
We can now prove

\begin{lemma}
    \label{le4.3}
    Let $1<p, q<\infty$, and let $\varphi$ be a bounded, compactly supported function on $\mathbb{C} $. Then $T_\varphi\in\mathcal{N}(F_\alpha^p,F_\alpha^q)$. Furthermore, for any $S\in\mathcal{L}(F_\alpha^q,F_\alpha^p)$, we have:
    \begin{equation}\label{eq7}
        \mathrm{Tr}_{F_\alpha^q}(T_\varphi S)=\frac{\alpha}{\pi}\int_\mathbb{C}\varphi(z)\widetilde{S}(z)dA(z).
    \end{equation}
\end{lemma}
\begin{proof}
  $T_\varphi\in\mathcal{N}(F_\alpha^p,F_\alpha^q)$ follows from Theorem \ref{th3.4}. Applying the pointwise estimates on the Fock space, we have
  \begin{equation*}
      |SK_z(u)|\leq\|SK_z\|_{\alpha,p}e^{\frac{\alpha}{2}|u|^2}\leq\|S\|e^{\frac{\alpha}{2}|z|^2}e^{\frac{\alpha}{2}|u|^2}.
  \end{equation*}
  Hence
    \begin{align*}
        &\int_\mathbb{C}|\varphi(u)|d\lambda_\alpha(u)\int_\mathbb{C}|K_u(z)||SK_z(u)|d\lambda_{\alpha}(z)\\
        &\leq C\int_\mathbb{C}|\varphi(u)|e^{\frac{\alpha}{2}|u|^2}d\lambda_\alpha(u)\int_\mathbb{C}|K_u(z)|e^{-\frac{\alpha}{2}|z|^2}dA(z)\\
        &=C\int_\mathbb{C}|\varphi(u)|dA(u)\\
        &<\infty.
    \end{align*}
   
Then the conclusion follows from Theorem \ref{th4.2}.
        
\end{proof}

We now present the proof of Theorem \ref{th1.2}:

    \begin{proof}[Proof of Theorem \ref{th1.2}]
        We first prove (2). Assume, for contradiction, that $C$ is not norm-dense in $\mathcal{K}(F_{\alpha}^{p},F_{\alpha}^{q})$. An application of the Hahn-Banach Theorem combined with Proposition \ref{pro2.5} implies the existence of a nonzero operator $S\in\mathcal{N}(F_{\alpha}^{q},F_{\alpha}^{p})$ satisfying

      $$\mathrm{Tr}(T_{\varphi}S)=0\qquad\text{for all }T_{\varphi}\in C.$$
      By Lemma \ref{le4.3}, this trace can be written as

$$\mathrm{Tr}(T_{\varphi}S)=\frac{\alpha}{\pi}\int_{\mathbb{C}}\varphi(z)\,\widetilde{S}(z)\,dA(z).$$
Since the equality holds for every continuous, compactly supported $\varphi$, it follows that $\widetilde{S}(z)=0$. Consequently, $S=0$, which contradicts the fact that $S$ is nonzero. This establishes statement (2).

The proof of (1) follows by a similar argument.
        \end{proof}
\bibliographystyle{abbrv}
\bibliography{ref}

\printaddress

\end{document}